\documentclass[11pt]{amsart}
\usepackage{amssymb,latexsym,amsmath,amscd,amsthm,amsfonts, enumerate}
\usepackage{multirow}
\usepackage{color}
\usepackage[all]{xy}
\usepackage{etex}
\usepackage{caption}
\usepackage{soul}
\usepackage{float}
\usepackage{titletoc}
\usepackage[normalem]{ulem}
\usepackage{graphicx}

\usepackage{tikz,tikz-cd}
\usepackage{mathrsfs}
\usepackage[all]{xy}
\usepackage{tikz}
\usepackage{extarrows}
\usepackage{tikz-cd}
\usetikzlibrary{calc}
\usetikzlibrary{matrix,arrows,decorations.pathmorphing}
\usetikzlibrary{snakes}
\usetikzlibrary{shapes.geometric,positioning}
\usetikzlibrary{arrows,decorations.pathmorphing,decorations.pathreplacing}
\usetikzlibrary{positioning,shapes,shadows,arrows,snakes}

\usepackage[colorlinks=true,pagebackref,hyperindex]{hyperref}
\newcommand\myshade{85}
\colorlet{mylinkcolor}{violet}
\colorlet{mycitecolor}{red}
\colorlet{myurlcolor}{cyan}

\hypersetup{
  linkcolor  = mylinkcolor!\myshade!black,
  citecolor  = mycitecolor!\myshade!black,
  urlcolor   = myurlcolor!\myshade!black,
  colorlinks = true,
}

\usepackage{tabularx}
\newcolumntype{E}{>{\hsize=0.5cm \centering\arraybackslash}X}%
\newcolumntype{C}[1]{>{\hsize=#1\hsize \centering\arraybackslash}X}%

\numberwithin{equation}{section}
\newtheorem{theorem}{Theorem}[section]
\newtheorem{proposition}[theorem]{Proposition}
\newtheorem{proposition-definition}[theorem]{Proposition-Definition}

\newtheorem{lemma}[theorem]{Lemma}
\newtheorem{theoremA}{Theorem}

\theoremstyle{definition}
\newtheorem{remark}[theorem]{Remark}
\newtheorem{example}[theorem]{Example}
\newtheorem{definition}[theorem]{Definition}

\definecolor{dark-green}{RGB}{14,150,2}
\definecolor{red}{RGB}{250,0,0}

\begin{document}

\title[Global dimension of a string algebra ]{Global dimension of a string algebra }

\author[Zheng Xin]{Zheng Xin}
\address{Zheng Xin,
	School of Mathematics and Statistics, Shaanxi Normal University, Xi'an 710062, China}
\email{xinzheng1314@yeah.net}

\author[LingChun Zhang]{LingChun Zhang*}
\address{LingChun Zhang,
	School of Mathematics and Statistics, Shaanxi Normal University, Xi'an 710062, China}
\email{zlcxpy912@163.com}

\thanks{
	}
\thanks{*Corresponding author}


\thanks{}

%


\begin{abstract}
In this paper, we characterize the global dimension of a string algebra by using combinatorial methods. 
Moreover, we establish a necessary and sufficient condition for when the global dimension of a string algebra is infinite.
\end{abstract}

\maketitle
\setcounter{tocdepth}{2} 

\tableofcontents

\section*{Introduction}\label{Introductions}

Global dimension plays an important role in the development of algebra, with extensive applications in ring theory, homological theory, representation theory and algebraic geometry. In recent years, the characterization of the global dimension of various finite-dimensional algebras has been a key focus of in-depth research in homological theory. The global dimensions of tree algebras \cite{H14}, gentle algebras \cite{C26,LGH24}, locally gentle algebras \cite{FOZ24} and almost gentle algebras \cite{WHL25} have been studied.
As a generalization of gentle algebras, string algebras are important finite-dimensional algebras introduced in \cite{BR87}. However, \cite{FGR21,FGR22} only established a sufficient condition for their infinite global dimension. 
Despite these notable advances in the study of related algebras, there still seems to be no systematic characterization of the global dimension for string algebras. 

In this paper, we use combinatorial methods to describe the projective resolution of simple modules over a string algebra. The minimal relation chains mentioned in the following results are given in Definition \ref{def:min-rel-chain}.

\begin{theoremA}[Theorem \ref{thm:gldim}]\label{Mtheorem:object}
	The global dimension of a string algebra is the maximum of the length of the minimal relation chains whose initial term is an arrow.
\end{theoremA}

As an application, we give a necessary and sufficient condition for a string algebra to have infinite global dimension, which is different from the sufficient condition in \cite{FGR21,FGR22}.
\begin{theoremA}[Proposition \ref{pro:gldim-infinite}]
	A string algebra has infinite global dimension if and only if there exists an infinite length minimal relation chain whose initial term is an arrow.
\end{theoremA}

\section*{Acknowledgments}
We would like to thank our supervisor, Prof. Wen Chang, for his guidance and assistance during our weekly seminar.  We would like to thank the refere for the helpful comments.

\section{Preliminaries}\label{Preliminaries}

In this paper, an algebra will be assumed to be basic with finite dimension over a base field $\mathbf{k}$ which is algebraically closed. 

\subsection{Path algebras}\label{subsection: path algebras}
In this subsection, we recall some basic definitions of path algebras.

A \emph{quiver} \( Q = (Q_0, Q_1, s, t) \) is a directed graph, which we always assume to be finite, with \( Q_0 \) the set of vertices, \( Q_1 \) the set of arrows, and \( s, t: Q_1 \to Q_0 \) two functions, sending an arrow to its start and target respectively. For an arrow \(\alpha\), let \(\alpha^{-1}\) be its \emph{forma inverse}, and \(s(\alpha^{-1}) = t(\alpha)\), \(t(\alpha^{-1}) = s(\alpha)\). Denote the set consisting of all inverses of arrows in $Q_{1}$ by \(Q_{1}^{-1}\).
A \emph{loop} (at a vertex \( v \in Q_0 \)) is an arrow \( \epsilon \in Q_1 \) with \( s(\epsilon) = t(\epsilon) (= v) \). A \emph{path} in \( Q \) of length \( n \geq 1 \) is a sequence \(p = a_1 \cdots a_n \) of arrows such that \( t(a_i) = s(a_{i+1}) \) for \( 1 \leq i \leq n - 1 \). The \emph{inverse} of the path $p$ is \(p^{-1} = a_n^{-1} \dots a_1^{-1}\), which satisfies \(s(a_i^{-1}) = t(a_i) = s(a_{i + 1}) = t(a_{i + 1}^{-1})\) for \(1 \leq i \leq n - 1\). In addition, to each vertex \( v \in Q_0 \), we associate a \emph{trivial path} \( e_v \) of length \( 0 \) with \( s(e_v) = v = t(e_v) \). If the number of arrows in a path is finite, then the path is called a \emph{finite path}.

By \( \mathbf{k}Q \) we denote the \emph{path algebra} of \( Q \). An ideal \( I \) of \( \mathbf{k}Q \) is called \emph{admissible} if \( R_Q^m \subseteq I \subseteq R_Q^2 \) for some \( m \geq 2 \), where \( R_Q \) denotes the ideal of \( \mathbf{k}Q \) generated by \( Q_1 \). The quotient algebra \( \mathbf{k}Q / I \) is finite dimensional whenever \( I \) is admissible. Arrows in a quiver are composed from left to right as follows: for arrows $a$ and $b$ we write $ab$ for the path from the source of $a$ to the target of $b$. In general, we consider right modules. We denote  by $P_{v} = e_{v}A$ the corresponding indecomposable projective $A-$module.

We use \(\mathbf{Pa}\) to represent the collection of all paths in \(\mathbf{k}Q/I\), which is precisely the set of all paths in Q that lie outside $I$. Denote by \(\mathbf{Pa}^{-1}\)  the set consisting of the inverses of all paths. Additionally, \(\mathbf{Pa}_{\geq l}\) (respectively, \(\mathbf{Pa}_{> l}\)) stands for the set of all paths in \(\mathbf{Pa}\) with length greater than or equal to a given non-negative integer $l$ (respectively, greater than $l$). 
Each element in \(A = \mathbf{k}Q/I\) can be uniquely expressed as a linear combination of elements in \(\mathbf{Pa}\), so we can regard \(\mathbf{Pa}\) as a basis for $A$.

The set \(\mathbf{M}\), consisting of \emph{maximal paths} in \(\mathbf{Pa}\), plays a crucial role in this entire theory. A path \(w = w_1 \cdots w_n\) in \(\mathbf{k}Q/I\) is \emph{maximal} in \(\mathbf{k}Q/I\), for every pair of arrows \(a, b \in Q_1\), neither $a w$ nor $wb$ is a path in \(\mathbf{k}Q/I\). Moreover, a non-trivial path $w$ in $Q $ belongs to \(\mathbf{Pa}\) if and only if it is a sub-path of a maximal path \(\widetilde{w}\) that is not contained in $I$. This maximal path takes the form \(\widetilde{w} = \hat{w} w \bar{w}\), where \(\hat{w}\) and \(\bar{w}\) are in \(\mathbf{Pa}\). For ease of future reference, the path \(\bar{w} \in \mathbf{Pa}\) is referred to as a \emph{right completion} of $w$.

For any element \( w \in \textbf{Pa}_{\geq 1} \) (with \( s(w) = v_{1} \) and \( t(w) = v_{2} \)), define the following morphism:
\[
p(w) : P(v_{2}) \to P(v_{1})
\]
\[
u \mapsto v = w u.
\]
Thus, any homomorphism
from $P(v_{2})$ to $P(v_{1})$ is associated to a linear combination of paths like $w$.

\subsection{String algebras}\label{subsection: string algebras}

In this subsection, we recall some basic definitions and propositions of  a string algebra.
The finite dimensional algebra $A$ is said to be \emph{monomial} if $I$ is generated by paths
of length at least two.

\begin{definition}\cite{AS87}\label{definition:string algebras}
We call an algebra $A=\mathbf{k}Q/I$ a \emph{string algebra}, if $Q=(Q_0,Q_1)$ is a finite quiver and $I$ is an admissible ideal of $\mathbf{k}Q$ satisfying the following conditions:
\begin{enumerate}[\rm(S1)]
 \item Each vertex in $Q_0$ is the source of at most two arrows and the target of at most two arrows.

 \item For each arrow $a$ in $Q_1$, there is at most one arrow $b$ such that  $ab\notin I$; at most one arrow $c$ such that  $ ca\notin I$.\label{S2} 

 \item $I$ is generated by paths of length at least two, i.e. $A$ is monomial.
\end{enumerate}
\end{definition}

Let $A = \mathbf{k}Q/I$ be a string algebra. 
A sequence \( W = w_1 \dots w_n \) is called a \emph{walk} (resp. \emph{generalized walk}) if each \( w_i \) belongs to \( Q_1 \cup Q_1^{-1} \) (resp. \( \textbf{Pa} \cup \textbf{Pa}^{-1} \)) and satisfies \( t(w_i) = s(w_{i + 1}) \) for all \( 1 \leq i \leq n-1 \).
The length of this sequence is \( n \), denoted as \( l(W) = n \). Obviously, \( s(W) = s(w_1) \) and \( t(W) = t(w_n) \). The inverse of a walk (generalized walk) is defined similarly to the inverse of a path. For a walk \( W = \alpha_1 \dots \alpha_n \) (\( \alpha_i \in Q_1 \cup Q_1^{-1}, 1 \leq i \leq n \)), its subwalk is \( \alpha_i \dots \alpha_j \) (\( 1 \leq i \leq j \leq n \)), and we call \( \alpha_i \) a \emph{letter} of the walk \( W \).

The \emph{concatenation} of two walks (generalized walks) \( W = w_1 \dots w_n \) and \( W' = w_1' \dots w_n' \) is defined as the walk \( WW' = w_1 \dots w_n w_1' \dots w_n' \), where \( s(w_1') = t(w_n) \).

The initial preliminary outcome that lies within the purview of our interests concerning string algebras serves to establish the fact that the right completion of an arrow $a$, which is denoted by \(\bar{a}\), possesses the property of uniqueness. 

\begin{lemma}\cite[Lemma 5]{FGR21}\label{lem:right-unique}
	Let \(A = \mathbf{k}Q/I\) be a string algebra and let \(a \in Q_1\). Then \(\bar{a}\) is unique.
\end{lemma}

Next, we introduce the structure of the kernel of the morphism \( p(w) \), where \( w \in \textbf{Pa}_{>0} \), which can be found in \cite{FGR21}.
According to the definition of a string algebra, there exist at most two distinct arrows \( a, b \) such that \( s(a) = s(b) = t(w) \). We know that the case where \( wa \neq 0 \) and \( wb \neq 0 \) cannot exist simultaneously, and at least one of \( wa \) and \( wb \) is equal to \( 0 \). Therefore, without loss of generality, we assume \( wa = 0 \). For a string algebra, there may exist \( wb \neq 0 \), or there may exist a path \( p = bp_1 \cdots p_r \) (where \( p_1, \cdots, p_r \in \textbf{Pa}_{>0} \)) such that \( wp = 0 \). If such a path \( p \) exists, we denote by \( b^* \) the shortest path satisfying this condition. Then, \( b^* \) is the shortest subpath of \( b\bar{b} \) and satisfies the property \( wb^* = 0 \).

\begin{lemma}\cite[Lemma 6]{FGR21}\label{Lem: Ker}
	If \( A = \mathbf{k}Q/I \) is a string algebra and \( w \in \textbf{Pa}_{>0} \), then the general structure of the kernel of the morphism \( p(w) \) is $\mathrm{ker}p(w) = aA \oplus b^*A$.
\end{lemma}

\begin{remark}
	(1) The proof of Lemma \ref{Lem: Ker} uses the fact that both \( u \) and \( b^* \) are subpaths of \( b\bar{b} \), and the right completion \( \bar{b} \) of \( b \) is unique.
	
	(2) For the convenience of subsequent discussion, we artificially distinguish \( a \) and \( b^* \) here, and denote \( a \) as the \emph{short kernel} of \( p(w) \), and \( b^* \) as the \emph{long kernel} of \( p(w) \).
\end{remark}

\begin{definition}\label{def:minimal sub-path}
	For \( w, u \in \textbf{Pa}_{>0} \), \( s(u) = t(w) \), and \( wu = 0 \), if any subpath \( u' (u' \neq u) \) of \( u \) does not satisfy \( wu' = 0 \), then the path \( u \) is called minimal for the path \( w \).
\end{definition}
\begin{remark}
	It follows from Definition \ref{def:minimal sub-path} that both the short kernel \( a \) and the long kernel \( b^* \) of the morphism \( p(w) \) are minimal for the path \( w \).
\end{remark}
 
The following conclusion provides a method for calculating the global dimension of finite-dimension $\mathbf{k}$-algebras.
\begin{lemma}\cite{We04}\label{lem:gldimA}
	The global dimension of finite-dimension $\mathbf{k}$-algebras is equal to the supremum of the projective dimensions of all its simple modules.
\end{lemma}

\section{The global dimension of a string algebra}\label{The global dimension of a string algebra}

In this section, we mainly rely on the relationship between the projective dimension of simple modules and the global dimension of a string algebra  to provide a combinatorial method for calculating the global dimension of a string algebra, and give a necessary and sufficient condition for the global dimension of a string algebra to be infinite.

\subsection{The projective dimension of simple modules }\label{subsection: The projective dimension of  simple modules}

In this subsection, we mainly give a formula for calculating the projective dimension of simple modules over a string algebra by using  combinatorial methods.

\begin{definition}\label{def:min-rel-chain}
	For a generalized non-trivial walk \( W = w_1 \cdot w_2 \cdots w_n \), we define
	\begin{enumerate}
		\item \( \mathscr{F}(W) = w_1 \), \( \mathscr{L}(W) = w_n \);
		\item \( \mathsf{Rel} := \left\{ W = w_1 \cdot w_2 \cdots w_n \,\middle|\, 
		w_i \in \textbf{Pa}_{>0}, \text{and}~w_{i + 1} \text{ is minimal for } w_i\right\} \);
		\item \( \mathsf{Rel}(\alpha) := \{ W \in \mathsf{Rel} \mid \mathscr{F}(W) = \alpha, \alpha \in Q_{1} \} \);
		\item \( \mathsf{Rel}(\alpha)_k := \{ W \in \mathsf{Rel}(\alpha) \mid l(W) = k,  k \in \mathbb{Z_{+}} \} \).
	\end{enumerate}
	We call the elements of the set Rel as $\emph{minimal relation chains}$, and those of the set Rel($\alpha$) as minimal relation chains whose initial term is an arrow.
\end{definition}

\begin{definition}
	For any vertex \( v \in Q_0 \), define a resolution \( P^\bullet \to S(v) \):
	
	\begin{equation} \label{equ:proj}
		\cdots \xrightarrow{\partial^{-i - 1}} P^{-i} \xrightarrow{\partial^{-i}} P^{-i + 1} \xrightarrow{\partial^{-i + 1}} \cdots \xrightarrow{\partial^{-2}} P^{-1} \xrightarrow{\partial^{-1}} P^0 \xrightarrow{\partial^0} S(v) \to 0  \tag{$\divideontimes$}.
	\end{equation}

	The modules in the resolution are:
\begin{gather*}
	P^0 = P(v), \quad P^{-1} = \bigoplus_{\alpha \in s^{-1}(v)} P(t(\alpha)), \\
	P^{-i} = \bigoplus_{\substack{\alpha \in s^{-1}(v) \\ W \in \mathsf{Rel}(\alpha)_i}} P(t(\mathscr{L}(W))), \quad i > 1.
\end{gather*}
	where \( W = w_1 \cdot w_2 \cdots w_n \), if \( \alpha \) does not exist, then \( P(t(\alpha)) = 0 \); If \( W \) does not exist, then \( P(t(\mathscr{L}(W))) = 0 \).
	
	The differentials in the resolution are:
	
	$\partial^0 = \pi$,
	 
	$\partial^{-1} = (p(\delta)~~p(\epsilon))$, $\delta \neq \epsilon$, $s(\delta) = s(\epsilon) = v$, if  $\delta$  does not exist, then  $p(\delta) = 0$, similarly for  $\epsilon$,
	
	For \( i > 1 \), \( \partial^{-i} = (\partial^{-i}_{jk})_{1 \leq j \leq 2^{i - 1}, 1 \leq k \leq 2^i} \), where
	\[
	\partial^{-i}_{jk} =
	\begin{cases}
		p(\mathscr{L}(W)), \, j = \mathcal{L}(p(w_{i - 1})), \, k = 2j - 1, \, \mathscr{L}(W) \text{ is the short kernel of } p(w_{i - 1}) \\
		\substack{\alpha \in s^{-1}(v) \\ W \in \mathsf{Rel}(\alpha)_i} \\ \\
		p(\mathscr{L}(W)), \, j = \mathcal{L}(p(w_{i - 1})), \, k = 2j, \, \mathscr{L}(W) \text{ is the long kernel of } p(w_{i - 1}) \\
		\substack{\alpha \in s^{-1}(v) \\ W \in \mathsf{Rel}(\alpha)_i} \\  \\
		0, \, \text{otherwise}
	\end{cases}
	\]
	
	Here, \( \pi \) is the natural projection, \( S(v) \) and \( P(v) \) are the simple module and projective module at the vertex \( v \) respectively. The function $\mathcal{L}$ is used to calculate the column index of elements in a matrix.
\end{definition}

\begin{remark}
	From the definition of string algebras, the structure of \( \ker p(w) \) (\( w \in \mathbf{Pa}_{>0} \)), and Definition \ref{def:minimal sub-path} , it can be known that the number of elements in the set \( \mathsf{Rel}(\alpha)_k \) is at most \( 2^k \). Therefore, in the above resolution \ref{equ:proj}, the module 
\begin{center}
		\( P^{-i} = \bigoplus\limits_{\substack{\alpha \in s^{-1}(v) \\ W \in \mathsf{Rel}(\alpha)_i}} P(t(\mathscr{L}(W))) \),  $i > 1 $
\end{center} 
	contains \( 2^k \) direct summands, including the zero module.
\end{remark}

The following result provides the formula for the projective resolution of simple modules.

\begin{lemma}\label{lem:proj-formula}
	Let \( A = \mathbf{k}Q/I \) be a string algebra, \( v \in Q_0 \), then the resolution (\ref{equ:proj}) is the minimal projective resolution of the simple module \( S(v) \).
\end{lemma}
\begin{proof}
	Since 
	\begin{center}
		\( P^0 = P(v) \), \( P^{-1} = \bigoplus\limits_{\alpha \in s^{-1}(v)} P(t(\alpha)) \), 
		
		\( P^{-i} = \bigoplus\limits_{\substack{\alpha \in s^{-1}(v) \\ W \in \mathsf{Rel}(\alpha)_i}} P(t(\mathscr{L}(W))) \) (\( i > 1 \)),
	\end{center}
	 for any \( i \geq 0 \), \( P^{-i} \) is a projective module.
	
	For \( i > 1 \), from the definition of \( \partial^{-i} = (\partial^{-i}_{jk})_{1 \leq j \leq 2^{i- 1 }, 1 \leq k \leq 2^i} \), we have:
	\[
	\partial^{-i}_{jk} = \begin{pmatrix}
		p(u_1) & p(u_2) & 0 & 0 & 0 & 0 & \cdots & 0 & 0 \\
		0 & 0 & p(u_3) & p(u_4) & 0 & 0 & \cdots & 0 & 0 \\
		\vdots & \vdots & \vdots & \vdots & \vdots & \vdots & \ddots & \vdots & \vdots \\
		0 & 0 & 0 & 0 & 0 & 0 & \cdots & p(u_{2^i - 1}) & p(u_{2^i})
	\end{pmatrix}_{2^{i- 1 } \times 2^i}
	\]
	In the \( j \)-th (\( 1 \leq j \leq 2^{i- 1 } \)) row of the differential \( \partial^{-i}_{jk} \), the possible non-zero elements are \( p(u_{2j - 1}) \) and \( p(u_{2j}) \), which are located in the \( (2j - 1) \)-th column and the \( 2j \)-th column respectively. Here, \( u_1, u_2, \ldots, u_{2^i - 1}, u_{2^i} \) correspond one-to-one to the elements \( W_1, W_2, \ldots, W_{2^i - 1}, W_{2^i} \) in \( \mathsf{Rel}(\alpha)_i \), and \( \mathscr{L}(W_m) = u_m \), \( 1 \leq m \leq 2^i \), \( m \in \mathbb{Z} \).
	\[
	\partial^{-i}: P^{-i} = \bigoplus\limits_{\substack{\alpha \in s^{-1}(v) \\ W \in \mathsf{Rel}(\alpha)_i}} P(t(\mathscr{L}(W))) \longrightarrow P^{-i + 1} = \bigoplus\limits_{\substack{\alpha \in s^{-1}(v) \\ W \in \mathsf{Rel}(\alpha)_{i - 1}}} P(t(\mathscr{L}(W))).
	\]
	For \( (m_1, m_2, \ldots, m_{2^i - 1}, m_{2^i})^T \in P^{-i} \), where \( m_1, m_2, \ldots, m_{2^i - 1}, m_{2^i} \) belong to different direct summands of \( P^{-i} \), we have:
	
	\[
	\begin{aligned}
		&\partial^{-i} \cdot (m_1, m_2, \ldots, m_{2^i - 1}, m_{2^i})^T \\
		&= \begin{pmatrix}
			p(u_1) & p(u_2) & 0 & 0 & 0 & 0 & \cdots & 0 & 0 \\
			0 & 0 & p(u_3) & p(u_4) & 0 & 0 & \cdots & 0 & 0 \\
			\vdots & \vdots & \vdots & \vdots & \vdots & \vdots & \ddots & \vdots & \vdots \\
			0 & 0 & 0 & 0 & 0 & 0 & \cdots & p(u_{2^i - 1}) & p(u_{2^i})
		\end{pmatrix} \begin{pmatrix}
			m_1 \\
			m_2 \\
			\vdots \\
			m_{2^i}
		\end{pmatrix} \\
		&= \begin{pmatrix}
			p(u_1)(m_1) + p(u_2)(m_2) \\
			p(u_3)(m_3) + p(u_4)(m_4) \\
			\vdots \\
			p(u_{2^i - 1})(m_{2^i - 1}) + p(u_{2^i})(m_{2^i})
		\end{pmatrix} \\
		&= \begin{pmatrix}
			u_1 m_1 + u_2 m_2 \\
			u_3 m_3 + u_4 m_4 \\
			\vdots \\
			u_{2^i - 1} m_{2^i - 1} + u_{2^i} m_{2^i}
		\end{pmatrix}.
	\end{aligned}
	\]
	
	Therefore, 
	\begin{center}
	\begin{align*}
		\text{Im}(\partial^{-i}) 
		&= (u_1 A \oplus u_2 A) \oplus (u_3 A \oplus u_4 A) \oplus \cdots \oplus (u_{2^i - 1} A \oplus u_{2^i} A) \\
		&= \bigoplus\limits_{\substack{\alpha \in s^{-1}(v) \\ W \in \mathsf{Rel}(\alpha)_i}} \mathscr{L}(W)A .
	\end{align*}
	\end{center}
	
	If \( (m_1, m_2, \cdots, m_{2^i - 1}, m_{2^i})^T \in \ker (\partial^{-i}) \), then we have
	\[
	\begin{cases}
		u_1 m_1 + u_2 m_2 = 0 \\
		u_3 m_3 + u_4 m_4 = 0 \\
		\vdots \\
		u_{2^i - 1} m_{2^i - 1} + u_{2^i} m_{2^i} = 0
	\end{cases}.
	\]
	Since \( u_1 m_1 \) and \( u_2 m_2 \) are linearly independent, we know \( u_1 m_1 = 0 \) and \( u_2 m_2 = 0 \). Similarly, we can obtain:
	\[
	u_3 m_3 = u_4 m_4 = \cdots = u_{2^i - 1} m_{2^i - 1} = u_{2^i} m_{2^i} = 0.
	\]
	That is, 
	\begin{center}
		\( m_1 \in \ker p(u_1) \), \( m_2 \in \ker p(u_2) \), \(\ldots\), \( m_{2^i - 1} \in \ker p(u_{2^i - 1}) \), \( m_{2^i} \in \ker p(u_{2^i}) \).
	\end{center} 
	Thus,
   \[
	\begin{aligned}
		\ker(\partial^{-i})  &\subseteq \ker p(u_1) \oplus \ker p(u_2) \oplus \cdots \oplus \ker p(u_{2^i - 1}) \oplus \ker p(u_{2^i}) \\
		&= \bigoplus\limits_{\substack{\alpha \in s^{-1}(v) \\ W \in \mathsf{Rel}(\alpha)_i}} \ker p(\mathscr{L}(W))\\ 
		& = \bigoplus\limits_{\substack{\alpha \in s^{-1}(v) \\ W \in \mathsf{Rel}(\alpha)_{i + 1}}} \mathscr{L}(W)A
	\end{aligned}.
	\]
	
	It is clear that 
\begin{center}
	\( \ker p(u_1) \oplus \ker p(u_2) \oplus \cdots \oplus \ker p(u_{2^i - 1}) \oplus \ker p(u_{2^i}) \subseteq \ker(\partial^{-i})  \). 
\end{center}
Therefore, 
     \[
		\ker(\partial^{-i}) = \bigoplus\limits_{\substack{\alpha \in s^{-1}(v) \\ W \in \mathsf{Rel}(\alpha)_{i + 1}}} \mathscr{L}(W)A. 
\] 
	Hence, for \( i > 1 \), \( \text{Im}(\partial^{-(i + 1)}) = \ker(\partial^{-i}) \).
	
	Next, we prove 
	\( \text{Im}(\partial^{-2}) = \ker(\partial^{-1}) \) and \( \text{Im}(\partial^{-1}) = \ker(\pi) \).
	
	It is known that \( \text{Im}(\partial^{-1}) = \bigoplus\limits_{\alpha \in s^{-1}(v)} \alpha A \), which is the submodule generated by all paths greater than 0 of \( P^0 = P(v) \), that is, \( \text{Im}(\partial^{-1}) = \ker(\pi) \). In addition, if \( v \) is the starting point of two different arrows \( \delta, \epsilon \), then \( \partial^{-1} = (p(\delta)~~~~p(\epsilon)) \). 
	For any element \( (x, y)^T \in \ker(\partial^{-1}) \), where \( x \in P(t(\delta)) \), \( y \in P(t(\epsilon)) \), then \( \partial^{-1}(x, y)^T = \delta x + \epsilon y = 0 \). And since \( \delta x, \epsilon y \) are linearly independent, \( \delta x = 0 \), \( \epsilon y = 0 \), thus \( x \in \ker p(\delta) \), \( y \in \ker p(\epsilon) \), that is
	\[
	\begin{aligned}
		\ker(\partial^{-1})  &\subseteq \ker p(\delta) \oplus \ker p(\epsilon) \\
		&= \bigoplus\limits_{\alpha \in s^{-1}(v)} \ker p(\alpha) \\
		&= \bigoplus\limits_{\substack{\alpha \in s^{-1}(v) \\ W \in \mathsf{Rel}(\alpha)_1}} \ker p(\mathscr{L}(W)) \\
		&= \bigoplus\limits_{\substack{\alpha \in s^{-1}(v) \\ W \in \mathsf{Rel}(\alpha)_2}} \mathscr{L}(W)A \\
		&= \text{Im}~(\partial^{-2}).
	\end{aligned}
	\]
	It is clear that \( \ker p(\delta) \oplus \ker p(\epsilon) \subseteq \ker(\partial^{-1}) \).
    Therefore, \( \ker(\partial^{-1}) = \text{Im}(\partial^{-2}) \).
	
	Next, we need to prove the minimality of resolution \ref{equ:proj}, that is, we need to prove that each differential \( \partial^{-i} (i \geq 1) \) and the natural projection \( \pi \) in the resolution are projective covers. 
	In other words, we need to prove 
    \( \ker(\pi) \subseteq P^0 \text{rad}(A) \)~~and~~~ \( \ker(\partial^{-i}) \subseteq P^{-i} \text{rad}(A) , i \geq 1 \) . 

	For any vertex \( v \in Q_0 \), we have 
	\begin{align*}
		P(v)\operatorname{rad}(A) &= \operatorname{rad}(P(v)) \\
		&= \bigoplus_{\alpha \in s^{-1}(v)} \alpha A .
	\end{align*}
	
	Then, when \( i \geq 2 \),
	\begin{align*}
		P^{-i}\operatorname{rad}(A) 
		&= \bigoplus_{\substack{\alpha \in s^{-1}(v) \\ W \in \mathsf{Rel}(\alpha)_i}} P(t(\mathscr{L}(W)))\operatorname{rad}(A) \\
		&= \bigoplus_{\substack{\alpha \in s^{-1}(v) \\ W \in \mathsf{Rel}(\alpha)_i \\ \beta \in s^{-1}(t(\mathscr{L}(W)))}} \beta A, \\[4pt]
		\ker(\partial^{-i}) 
		&= \bigoplus_{\substack{\alpha \in s^{-1}(v) \\ W \in \mathsf{Rel}(\alpha)_{i + 1}}} \mathscr{L}(W)A \\
		&\subseteq P^{-i}\operatorname{rad}(A).
	\end{align*}
	
	In addition,
	\begin{align*}
		P^{-1}\operatorname{rad}(A) 
		&= \bigoplus_{\alpha \in s^{-1}(v)} P(t(\alpha))\operatorname{rad}(A) \\
		&= \bigoplus_{\substack{\alpha \in s^{-1}(v) \\ \beta \in s^{-1}(t(\alpha)) }} \beta A, \\[4pt]
		\ker(\partial^{-1})
		&= \bigoplus_{\substack{\alpha \in s^{-1}(v) \\ W \in \mathsf{Rel}(\alpha)_2}} \mathscr{L}(W)A \\
		&\subseteq P^{-1}\operatorname{rad}(A).
	\end{align*}
	
	Moreover,
	\begin{align*}
		P^0\operatorname{rad}(A) 
		&= P(v)\operatorname{rad}(A) \\
		&= \bigoplus_{\alpha \in s^{-1}(v)} \alpha A \\
		&= \ker(\pi).
	\end{align*}
	
	Therefore, the resolution \ref{equ:proj} is the minimal projective resolution of the simple module \( S(v) \).

\end{proof}

\begin{remark}
 The possible non-zero elements $u_{1}, u_{2}$ in the differential $\partial^{-i}_{jk}$ are the short kernel and long kernel of $p(\mathscr{L}(W'_{1}))$ associated with $W'_{1} \in \mathop{\mathsf{Rel}}\limits_{\alpha \in s^{-1}(v)}(\alpha)_{i-1}$, respectively. Similarly, $u_{3}, u_{4}$ correspond to $W'_{2}$, and $u_{2^{i}-1}, u_{2^{i}}$ correspond to $W'_{2^{i-1}}$. 
	 
	 In addition, the following conditions apply:
	 \begin{itemize}
	 	\item If $u_{k}$ does not exist (for any $k \in \{1, \dots, 2^{i}\}$), then $p(u_{k}) = 0$.
	 	\item If $W'_{m}$ does not exist, then $p(u_{2m-1}) = p(u_{2m}) = 0$ (for any $m \in \{1, \dots, 2^{i-1}\}$).
	 \end{itemize}
	 
\end{remark}

\subsection{The global dimension of a string algebra}\label{subsection: The global dimension of a string algebra}
According to the relationship between the global dimension of a string algebra and the projective dimension of simple modules over it, on the basis of Lemma \ref{lem:proj-formula}, the main conclusion of this chapter is given below.

\begin{theorem}\label{thm:gldim}
	The global dimension of a string algebra \( A = \mathbf{k}Q/I \) is
	\[
	\textbf{gl.dim}~A = \max\limits_{\substack{\alpha \in Q_1 \\ W \in \mathsf{Rel}(\alpha)}} l(W).
	\]
\end{theorem}
\begin{proof}
	 By Lemma \ref{lem:proj-formula}, for any \( v \in Q_0 \), the projective dimension of the corresponding a simple module is equal to the length of the longest element \( W \) in the set \( \underset{\alpha \in s^{-1}(v)}{\mathsf{Rel}}(\alpha) \), that is,
	\[
	\textbf{proj.dim} S(v) = \max\limits_{\substack{\alpha \in s^{-1}(v) \\ W \in \mathsf{Rel}(\alpha)}} l(W).
	\]
	
	By Lemma \ref{lem:gldimA}, We can directly obtain the global dimension of $A$, that is,
	\[
	\textbf{gl.dim}~A = \max\limits_{v \in Q_0} \textbf{proj.dim} S(v).
	\]
	
	\noindent Thus, we have
	\[
	\textbf{gl.dim}~A = \max\limits_{v \in Q_0} \max\limits_{\substack{\alpha \in s^{-1}(v) \\ W \in \mathsf{Rel}(\alpha)}} l(W) = \max\limits_{\substack{\alpha \in Q_1 \\ W \in \mathsf{Rel}(\alpha)}} l(W).
	\]
\end{proof}

By Theorem \ref{thm:gldim}, we can easily derive the necessary and sufficient condition for \( \text{gldim}\,A = \infty \).
\begin{proposition}\label{pro:gldim-infinite}
	For a string algebra \( A = \mathbf{k}Q/I \), the global dimension of \( A \) is infinite if and only if there exists \( W \in \mathop{\mathsf{Rel}}\limits_{\alpha \in Q_{1}} (\alpha) \) such that \( l(W) = \infty \).
\end{proposition}
\begin{proof}
	For any simple module \( S(v) \) (associated with vertex \( v \in Q_0 \)) over a string algebra \( A \), the projective dimension \( \textbf{proj.dim}\,S(v) \) is determined by the maximum number of the length of the elements in \( \mathop{\mathsf{Rel}}\limits_{\alpha \in s^{-1}(v)} (\alpha) \):
	\[
	\textbf{proj.dim}\,S(v) = \max\limits_{\substack{\alpha \in s^{-1}(v) \\ W \in \mathsf{Rel}(\alpha)}} l(W),
	\]
	where \( s^{-1}(v) \subseteq Q_1 \) denotes the set of arrows with source \( v \), and \( l(W) \) is the length of \( W \).

	The global dimension of \( A \) is defined as the maximum projective dimension of all simple modules over \( A \):
	\[
	\textbf{gl.dim}\,A = \max\limits_{v \in Q_0} \textbf{proj.dim}\,S(v).
	\]

	Suppose there exists \( \alpha \in Q_1 \) and \( W \in \mathsf{Rel}(\alpha) \) such that \( l(W) = \infty \). Let \( v \) be the source of \( \alpha \) (i.e., \( \alpha \in s^{-1}(v) \)). By the formula for \( \textbf{proj.dim}\,S(v) \), the infinite length of \( W \) implies:
	\[
	\textbf{proj.dim}\,S(v) = \max\limits_{\substack{\alpha \in s^{-1}(v) \\ W \in \mathsf{Rel}(\alpha)}} l(W) = \infty.
	\]
	By the definition of global dimension, it follows that:
	\[
	\textbf{gl.dim}\,A = \max\limits_{v \in Q_0} \textbf{proj.dim}\,S(v) = \infty.
	\]

	Conversely, suppose \( \textbf{gl.dim}\,A = \infty \). By Lemma \ref{lem:gldimA}, there exists a simple module \( S(v) \) with \( \textbf{proj.dim}\,S(v) = \infty \). From the formula for \( \textbf{proj.dim}\,S(v) \), this implies:
	\[
	\max\limits_{\substack{\alpha \in s^{-1}(v) \\ W \in \mathsf{Rel}(\alpha)}} l(W) = \infty.
	\]
	Thus, there must exist  \( W \in \mathop{\mathsf{Rel}}\limits_{\alpha \in Q_{1}}(\alpha) \) such that \( l(W) = \infty \).
\end{proof}

\section{Example}

In this section, we give a concrete example to illustrate the result we have obtained.

\begin{example}

Consider the quiver $Q$
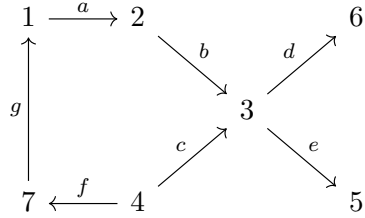
	\begin{center}
		\begin{tikzcd}
			1 \arrow[r, "a"]  & 2 \arrow[rd, "b"]                 &                                   & 6 \\
			&                                   & 3 \arrow[rd, "e"] \arrow[ru, "d"] &   \\
			7 \arrow[uu, "g"] & 4 \arrow[ru, "c"] \arrow[l, "f"'] &                                   & 5
		\end{tikzcd}
	\end{center}
	and the admissible ideal $I = <fga, ab, bd, be, cd, ce>$.  Then we have
	\[
	\begin{aligned}
		\mathsf{Rel}(a) &= \{a, a \cdot b, a \cdot b \cdot d, a \cdot b \cdot e\}, \\
		\mathsf{Rel}(b) &= \{b, b \cdot d, b \cdot e\}, \\
		\mathsf{Rel}(c) &= \{c, c \cdot d, c \cdot e\}, \\
		\mathsf{Rel}(d) &= \{d\}, \\
		\mathsf{Rel}(e) &= \{e\}, \\
		\mathsf{Rel}(f) &= \{f, f \cdot ga, f \cdot ga \cdot b, f \cdot ga \cdot b \cdot d, f \cdot ga \cdot b \cdot e\}, \\
		\mathsf{Rel}(g) &= \{g\},
	\end{aligned}
	\]
	
	Now, consider the projective resolution of simple modules over the string algebra \( A = \mathbf{k}Q/I \). It is easy to see that the simple modules \( S(5) \) and \( S(6) \) are projective modules, i.e., their projective dimensions are 0. We consider the projective resolutions of the remaining simple modules.\\
	\textbf{(1)  Projective resolution of \( S(1) \)}
	
	Projective modules in the resolution of \( S(1) \):
	\[
	\begin{aligned}
		P^0    &=  P(1), \\
		P^{-1} &=  P(2) \oplus 0 = P(2), \\
		P^{-2} &=  P(3) \oplus 0 \oplus 0 \oplus 0 = P(3) \oplus 0, \\
		P^{-3} &=  P(6) \oplus P(5) \oplus 0 \oplus 0 \oplus 0 \oplus 0 \oplus 0 \oplus 0  = P(6) \oplus P(5) \oplus 0 \oplus 0, 
	\end{aligned}
	\]

	Differentials in the resolution:
	\[
	\partial^{-1} = \begin{pmatrix} p(a) & 0 \end{pmatrix} \longleftrightarrow p(a),
	\]
	\[
	\partial^{-2} = \begin{pmatrix} p(b) & 0 & 0 & 0 \\ 0 & 0 & 0 & 0 \end{pmatrix} \longleftrightarrow \begin{pmatrix} p(b) & 0 \end{pmatrix},
	\]
	\[
	\partial^{-3} = \begin{pmatrix} p(d) & p(e) & 0 & 0 & 0 & 0 & 0 & 0 \\ 0 & 0 & 0 & 0 & 0 & 0 & 0 & 0 \\ 0 & 0 & 0 & 0 & 0 & 0 & 0 & 0 \\ 0 & 0 & 0 & 0 & 0 & 0 & 0 & 0 \end{pmatrix} \longleftrightarrow \begin{pmatrix} p(d) & p(e) & 0 & 0 \\ 0 & 0 & 0 & 0 \end{pmatrix}.
	\]
	
	The projective resolution of \( S(1) \) is
	\[
	0 \to P(6) \oplus P(5) \oplus 0 \oplus 0 \xrightarrow{\partial^{-3}} P(3) \oplus 0 \xrightarrow{\partial^{-2}} P(2) \xrightarrow{\partial^{-1}} P(1) \to S(1)  \to 0,
	\]
	then the projective dimension of the simple module \( S(1) \) is 
	\begin{center}
		\( \textbf{proj.dim} S(1) = \max\limits_{\substack{W \in \mathsf{Rel}(a)}} l(W) = 3 \).
	\end{center}
	\textbf{(2)  Projective resolution of \( S(2) \)}
	
	Projective modules in the resolution of \( S(2) \):
	\[
	\begin{aligned}
		P^0    &=  P(2), \\
		P^{-1} &=  P(3) \oplus 0 = P(3), \\
		P^{-2} &=  P(6) \oplus P(5) \oplus 0 \oplus 0 = P(6) \oplus P(5), \\ 
	\end{aligned}
	\]
	Differentials in the resolution:
	\[
	\partial^{-1} = \begin{pmatrix} p(b) & 0 \end{pmatrix} \longleftrightarrow p(b),
	\]
	\[
	\partial^{-2} = \begin{pmatrix} p(d) & p(e) & 0 & 0 \\ 0 & 0 & 0 & 0 \end{pmatrix} \longleftrightarrow \begin{pmatrix} p(d) & p(e) \end{pmatrix}.
	\]
	
	The projective resolution of \( S(2) \) is
	\[
	0 \to P(6) \oplus P(5) \xrightarrow{\begin{pmatrix} p(d) & p(e) \end{pmatrix}} P(3) \xrightarrow{p(b)} P(2) \to S(2) \to 0,
	\]
	then the projective dimension of the simple module \( S(2) \) is 
	\begin{center}
		\( \textbf{proj.dim} S(2) = \max\limits_{\substack{W \in \mathsf{Rel}(b)}} l(W) = 2 \).
	\end{center}
	\textbf{(3)  Projective resolution of \( S(3) \)}
	
	Projective modules in the resolution of \( S(3) \): 
	\[
	\begin{aligned}
		P^0    &=  P(3), \\
		P^{-1} &=  P(6) \oplus P(5), \\
	\end{aligned}
	\]
	Differentials in the Resolution:
	\[
	\partial^{-1} = \begin{pmatrix} p(d) & p(e) \end{pmatrix}.
	\]
	
	\noindent The projective resolution of \( S(3) \) is:
	\[
	0 \to P(6) \oplus P(5) \xrightarrow{\begin{pmatrix} p(d) & p(e) \end{pmatrix}} P(3) \to S(3) \to 0,
	\]
	then the projective dimension of the simple module \( S(3) \) is 
\begin{center}
		\( \textbf{proj.dim}S(3) = \max\limits_{\substack{W \in \mathsf{Rel}(d) \cup \mathsf{Rel}(e)}}l(W) = 1 \).
\end{center}
	\textbf{(4)  Projective resolution of \( S(4) \)}
	
	Projective modules in the resolution: 
	\[
	\begin{aligned}
		P^0    &=  P(4), \\
		P^{-1} &=  P(3) \oplus P(7), \\
		P^{-2} &=  P(6) \oplus P(5) \oplus 0 \oplus P(2), \\ 
		P^{-3} &=  0 \oplus 0 \oplus 0 \oplus 0 \oplus 0 \oplus 0 \oplus P(3) \oplus 0,\\ 
		P^{-4} &=  0 \oplus 0 \oplus 0 \oplus 0 \oplus 0 \oplus 0 \oplus 0 \oplus 0 \oplus 0 \oplus 0 \oplus 0 \oplus 0 \oplus P(6) \oplus P(5) \oplus 0 \oplus 0,\\
	\end{aligned}
	\]
	 Differentials in the resolution:
	\[
	\partial^{-1} = \begin{pmatrix} p(c) & p(f) \end{pmatrix},
	\]
	\[
	\partial^{-2} = \begin{pmatrix} p(d) & p(e) & 0 & 0 \\ 0 & 0 & 0 & p(ga) \end{pmatrix},
	\]
	\[
	\partial^{-3} = \begin{pmatrix}0 & 0 & 0 & 0 & 0 & 0 & 0 & 0 \\ 0 & 0 & 0 & 0 & 0 & 0 & 0 & 0 \\ 0 & 0 & 0 & 0 & 0 & 0 & 0 & 0 \\ 0 & 0 & 0 & 0 & 0 & 0 & p(b) & 0 \end{pmatrix},
	\]
	\[
	\partial^{-4} = \begin{pmatrix} 0 & 0 & \cdots & 0 & 0 & 0 & 0 \\ 0 & 0 & \cdots & 0 & 0 & 0 & 0 \\ \vdots & \vdots & \ddots & \vdots & \vdots & \vdots & \vdots \\ 0 & 0 & \cdots & p(d) & p(e) & 0 & 0 \\ 0 & 0 & \cdots & 0 & 0 & 0 & 0 \end{pmatrix}_{8 \times 16}.
	\]
	
	\noindent The projective resolution of \( S(4) \) is:
	\[
	0 \to P^{-4} \xrightarrow{\partial^{-4}} P^{-3} \xrightarrow{\partial^{-3}} P^{-2} \xrightarrow{\partial^{-2}} P^{-1} \xrightarrow{\partial^{-1}} P(4) \to S(4) \to 0,
	\]
	then the projective dimension of the simple module \( S(4) \) is 
	\begin{center}
		\( \textbf{proj.dim}S(4) = \max\limits_{\substack{W \in \mathsf{Rel}(c) \cup \mathsf{Rel}(f)}} l(W) = 4 \).
	\end{center}
	\textbf{(5)  Projective resolution of \( S(7) \)}
	
	The projective modules in the resolution: 
	\[
	\begin{aligned}
		P^0    &=  P(7), \\
		P^{-1} &=  P(1) \oplus 0 = P(1), \\
	\end{aligned}
	\]
	Differentials in the resolution:
	\[
	\partial^{-1} = \begin{pmatrix} p(g) & 0 \end{pmatrix} \longleftrightarrow p(g).
	\]
	
	The projective resolution of \( S(7) \) is
	\[
	0 \to P(1) \xrightarrow{p(g)} P(7) \to S(7) \to 0,
	\]
	then the projective dimension of the simple module \( S(7) \) is 
	\begin{center}
		\( \textbf{proj.dim}S(7) = \max\limits_{\substack{W \in \mathsf{Rel}(g)}} l(W) = 1 \).
	\end{center}
	
	Moreover, since the global dimension of the string algebra \( A \) is equal to the maximum value of the projective dimensions of all its simple modules, it can be known that the global dimension of the string algebra \( A \) in this example is 4.

\end{example}

\end{document}